\documentclass[11pt,reqno,a4paper]{amsart}
\usepackage{amsmath,amssymb,amsfonts}
\usepackage{amsthm}
\usepackage{array}
\usepackage{resizegather}
\usepackage{array,longtable}
\newcolumntype{C}{>{$}c<{$}}
\newcolumntype{L}{>{$}l<{$}}
\usepackage{color}
\definecolor{myorange}{RGB}{235,140,30}
\definecolor{mygreen}{RGB}{28,172,0} 
\definecolor{mylilas}{RGB}{170,55,241}
\definecolor{azulescuro}{RGB}{0,0,102}
\usepackage[utf8]{inputenc} 
\usepackage{graphicx}
\usepackage{setspace}
\usepackage{epstopdf}
\usepackage{soul}
\usepackage{amsmath}
\usepackage{amsfonts}
\usepackage{stackengine}
\usepackage{amssymb}
\usepackage{amsthm}
\usepackage{verbatim}
\usepackage{array}
\usepackage{listings}
\usepackage{float}
\usepackage{mathtools}
\usepackage{xcolor, soul}
\usepackage{relsize}
\usepackage[shortlabels]{enumitem}
\usepackage{booktabs, isotope}
\usepackage{caption}
\usepackage{subcaption}
\usepackage{wrapfig}
\allowdisplaybreaks
\usepackage[toc,page]{appendix}
\usepackage{multirow}
\usepackage{lmodern}
\usepackage[hidelinks]{hyperref}
\mathtoolsset{showonlyrefs}

\usepackage{dsfont}
\usepackage{tcolorbox}

\theoremstyle{plain}
\newtheorem{thm}{Theorem}[section]
\newtheorem{thmx}{Theorem}
\newtheorem{corx}[thmx]{Corollary}

\newtheorem{lemma}[thm]{Lemma}
\newtheorem{proposition}[thm]{Proposition}
\newtheorem{corollary}[thm]{Corollary}

\theoremstyle{definition}

\newtheorem{conjecture}[thm]{Conjecture}

\newtheorem{remark}{Remark}[section]

\numberwithin{equation}{section}

\newcommand{\bo}{{\rm O}}
\newcommand{\ds}{\displaystyle}
\newcommand{\ch}{\mathcal{C}_{h}}
\newcommand{\dint}{\ds\int}
\newcommand{\dsum}{\ds\sum}

\newcommand{\eqskip}{ \vspace*{2mm}\\ }
\newcommand{\R}{\mathbb{R}}
\newcommand{\N}{\mathbb{N}}
\newcommand{\Z}{\mathbb{Z}}

\newcommand{\eln}{\ell_{n}}
\newcommand{\elm}{\ell_{m}}

\newcommand{\SI}{\mathbb{S}^{1}\times I_{h}}

\newcommand{\fr}[2]{\frac{\ds #1}{\ds #2}}

\newcommand{\liyauA}{\sqrt{\pi^2+1}-\pi}
\newcommand{\liyauB}{\left(\pi+\frac{1}{\sqrt{\pi}}\right)\left(\fr{2}{3}\right)^{3/2}}

\makeatletter
\@namedef{subjclassname@2020}{%
	\textup{2020} Mathematics Subject Classification}
\makeatother

\usepackage{etoolbox}
\patchcmd{\section}{\scshape}{\bfseries}{}{}
\makeatletter
\renewcommand{\@secnumfont}{\bfseries}
\makeatother

\begin{document}

\title{P\'{o}lya's conjecture on $\mathbb{S}^1 \times \R$}

\author[P. Freitas]{Pedro Freitas}
\author[R. Z. Wang]{Rui Zhu Wang}

\address{Grupo de F\'{\i}sica Matem\'{a}tica, Instituto Superior T\'{e}cnico, Universidade de Lisboa, Av. Rovisco Pais, 1049-001 Lisboa, Portugal}
\email{pedrodefreitas@tecnico.ulisboa.pt}
\address{Departamento de Matem\'{a}tica, Instituto Superior T\'{e}cnico, Universidade de Lisboa, Av. Rovisco Pais, 1049-001 Lisboa, Portugal}
\email{ruiwang.pt@gmail.com}

\date{\today}

\begin{abstract}
    We study the area ranges where the two possible isoperimetric domains on the infinite cylinder $\mathbb{S}^{1}\times \R$, namely, geodesic disks and cylindrical strips of the form $\mathbb{S}^1\times [0,h]$, satisfy P\'{o}lya's conjecture. In the former case, we provide an upper bound on the maximum value of the radius for which the
	conjecture may hold, while in the latter we fully characterise the values of $h$ for which it does hold for these strips. As a consequence,
	we determine a necessary and sufficient condition for the isoperimetric domain on $\mathbb{S}^{1}\times \R$ corresponding to a given area to satisfy P\'{o}lya's conjecture. In the case of the cylindrical strip, we
 	also provide a necessary and sufficient condition for the Li-Yau inequalities to hold.
\end{abstract}
\maketitle

\section{Introduction}
P\'{o}lya's conjecture for the Dirichlet Laplacian on a domain states that all its eigenvalues are greater than or equal to the first term in the corresponding Weyl asymptotics. It was formulated by P\'{o}lya in the first edition of~\cite{poly1} in 1954, who then proved it in 1961 for the case of domains that tile the plane~\cite{poly2}.
Apart from the direct extension of this result to higher-dimensional Euclidean domains by Urakawa~\cite{urak83}, the conjecture is only known to hold for certain Cartesian products~\cite{lapt,fs,hw}, for
Euclidean balls and sectors~\cite{flps}, and for certain families of domains tiling a given domain~\cite{fs,frlapa}. The result in~\cite{fs} essentially states that if a domain $\Omega$ satisfies certain geometric conditions and may
be tiled by an arbitrary number of $p$ equal isometric copies of a domain $\Omega_{p}$, then there exists $p_{0}$
large enough such that $\Omega_{p}$ satisfies P\'{o}lya's conjecture for all $p$ greater than or equal to $p_{0}$.
This result is essentially qualitative, but, on the other hand, it is not restricted to Euclidean domains, thus
showing that the conjecture may also be satisfied for certain families of domains on
manifolds. To the best of our knowledge and until then, for the Dirichlet problem on non-Euclidean domains this
had only been known to hold for the two-dimensional hemisphere~\cite{bb,blps,fms}.

One of the purposes of the present paper is to contribute to the understanding of what differences may be expected
when considering P\'{o}lya's conjecture on manifolds. To illustrate this behaviour, one of the examples given in~\cite{fs}
is that of a cylindrical surface of the form $\ch=\mathbb{S}^{1}\times I_{h}$, where $I_{h}$ is an open interval of length $h$,
and for which it is concluded that P\'{o}lya's conjecture is satisfied provided $h$ is sufficiently small
\cite[Example III.D]{fs}. On the other hand, a simple argument based on the behaviour of the first eigenvalue was
presented there showing that for large enough $h$ this will no longer be the case -- see below and also
Section~\ref{sec-upp}. This shows that topological issues may play a role, with the consequence that
now tiling domains no longer necessarilly satisfy the conjecture, not even in a weaker form -- see
Proposition~\ref{noweaker}. As far as we can tell, the only instance where tiling domains had been
considered previously on cylindrical surfaces were first in a paper by Hersch, who also considered
domains on the torus~\cite{hers}, and also in a paper by the first author~\cite{frei}. In either case,
only a finite number of tiles were considered.

Apart from illustrating the different possible behaviours mentioned above, another purpose of the present article
is to provide an example of a quantitative version of the result in~\cite{fs} by characterising precisely the values
of $h$ for which P\'{o}lya's conjecture is satisfied for $\ch$. In this case the conjecture reads as
\[
 \lambda_{k}\left( \ch  \right) \geq \fr{4k\pi}{\left| \ch  \right| } = \fr{2k}{h},
\]
where $\lambda_{k}\left( \ch  \right)$ and $\left| \ch  \right|$ denote the $k^{\rm th}$ Dirichlet eigenvalue
and the area of $\ch $, respectively. Since the first eigenvalue is given by $\lambda_{1}\left(\ch \right)=\pi^2/h^2$, it is straightforward to see that this eigenvalue cannot satisfy the conjecture for $h$ larger than
$\pi^2/2$. However, it is not clear whether the conjecture holds for {\it all} $h$ smaller than or equal to this value and, in
fact, we will show that there are indeed two disjoint open intervals for smaller values of $h$ where the conjecture
fails. More precisely, our main result is the following.
\begin{thmx}\label{mainthm}
    Pólya's conjecture on $C_{h}=\SI$ is satisfied if and only if $h$ belongs to
    \[ \left(0,8-\sqrt{64-4\pi^2}\right] \cup \left[2+\sqrt{4-\frac{\pi^2}{4}},13-\sqrt{169-9\pi^2}\right]
    \cup \left[\frac{13}{4}+\sqrt{\frac{169}{16}-\pi^2},\frac{\pi^2}{2}\right].
    \]
    On the intervals
    \[
     \left(8-\sqrt{64-4\pi^2},2+\sqrt{4-\frac{\pi^2}{4}}\right) \approx (3.0481, 3.2380)
    \]
and
    \[
     \left(13-\sqrt{169-9\pi^2},\frac{13}{4}+\sqrt{\frac{169}{16}-\pi^2}\right) \approx (4.0460, 4.0824)
    \]
    the conjecture fails for $\lambda_{8}\left( \ch  \right)$ and $\lambda_{13}\left( \ch  \right)$,
    respectively, and only for these eigenvalues in each case. From $\pi^2/2$ onwards the conjecture fails for
    $\lambda_{1}\left( \ch  \right)$ and then for some higher eigenvalues as $h$ is increased further.
\end{thmx}
As a direct corollary of this result, we have the precise values of the area for which the isoperimetric domains on
the infinite cylinder $\mathbb{S}^{1}\times \R$ satisfy P\'{o}lya's inequality. For two-dimensional manifolds,
given a value for the area an isoperimetric domain is a domain with that area with the smallest possible perimeter.
In the case of the infinite cylinder, isoperimetric domains are either a geodesic disk for areas smaller than or
equal to $4\pi$, or, for larger areas, a cylindrical strip $\ch$ bounded by two parallel circles as described
above~\cite{hhm}.
\begin{corx}\label{cor-isop}
 For a given value of the area, the corresponding isoperimetric domains on $\mathbb{S}^{1}\times\R$ satisfy P\'{o}lya's conjecture if and only if this area is in the range
 \[
 \begin{array}{ll}
  \left(0,2\pi\left( 8-\sqrt{64-4\pi^2} \right)\right] \cup \left[2\pi\left(2+\sqrt{4-\fr{\pi^2}{4}}\right),
  2\pi\left( 13-\sqrt{169-9\pi^2}\right)\right] & \eqskip
   \hspace*{6cm} \cup \left[2\pi\left(\fr{13}{4}+\sqrt{\fr{169}{16}-\pi^2}\right),\pi^3\right]
 \end{array}
 \]
\end{corx}
\begin{proof}
 For a disk of area smaller than or equal to $4\pi$, where the transition between disks and strips takes place,
 its radius is smaller than or equal to two, smaller than $\pi$, and we know that in this range P\'{o}lya's conjecture
 is satisfied for disks~\cite{flps}. Since this is within the range of the first interval where P\'{o}lya's conjecture
 holds for the strip, the result follows.
\end{proof}
Although as we have seen P\'{o}lya's conjecture fails for some values of $h$ below $\pi^2/2$, we actually have that
the corresponding averaged version compatible with Weyl asymptotics, known as the Li-Yau inequalities~\cite{liyau},
holds in this range. In fact, in this case the values of $h$ for which this happens may be extended up to $\pi^2$.
\begin{corx}\label{cor-liyau} The inequalities
\begin{equation}\label{liyau}
\frac{1}{k}\sum_{i=1}^k \lambda_i\geq \frac{k}{h},\quad k=1,2,\dots
\end{equation}
hold for all $k$ if and only if $h\in(0,\pi^2]$.
\end{corx}
\noindent This result has implications for the satisfiability of P\'{o}lya's conjecture for Cartesian products
involving $\mathbb{S}^{1}\times I_{h}$, such as
$\mathbb{S}^{1}\times I_{h_{1}}\times \dots \times I_{h_{m}}$, which will be briefly discussed in Section~\ref{other}.

The proof of Theorem~\ref{mainthm} follows from a careful analysis of the integer lattice point counting problem associated
with the eigenvalues on $\ch $. In Section~\ref{bounds} we derive estimates for the
eigenvalues which are sufficiently sharp to allow us to reduce the full problem to that of deciding whether or not
P\'{o}lya's conjecture holds on a bounded set in the $(h,\lambda)-$plane. We then devise an algorithm, whose
implementation is given in Appendix~\ref{apx:impl}, to carry out a rigorous search within ths set for possible cases
where the conjecture fails. Section~\ref{liyau-sec} addresses the proof of the Li-Yau inequalities, which follows
from the bounds for individual eigenvalues and two basic lemmas. In Section~\ref{other}, we consider some further
examples, including other domains in $\mathbb{S}^{1}\times \R$ such as
disks with radius larger than $\pi$, and higher-dimensional Cartesian products like $\mathbb{S}^{n}\times I_{h}$
and the already mentioned $\mathbb{S}^{1}\times {\ds\prod_{j=1}^{n}} I_{h_{j}}$.

The specific values of $h$ for which P\'{o}lya's conjecture for $\ch$ holds or fails depend on several factors, some
of which are not obvious a priori. Two ingredients which might appear to be relevant at first sight are
simply-connectedness and the isoperimetric inequaltiy -- for the latter, see~\cite{hhm} for the two-dimensional
cylinder and~\cite{pedr} for a study of the isoperimetric problem in infinite spherical cylinders.
More precisely, all the examples given here where P\'{o}lya's conjecture fails are neither simply-connected nor do
they satisfy the Euclidean isoperimetric inequality. It
is also clear that simply-connected domains which tile $\mathbb{S}^{1}\times \R$ also tile $\R^{2}$ and
will thus satisfy P\'{o}lya's conjecture. However, should a non-simply connected planar domain satisfy the conjecture,
there will also be a copy, scaled down if necessary, which will satisfy it as a subset of $\mathbb{S}^{1}\times \R$.
Moreover, by picking a strip $\ch$ with sufficiently large $h$ so that its first eigenvalue does not satisfy the
conjecture, it will be possible to modify its boundary so that the first eigenvalue will continue not satisfying the
conjecture, while keeping the area fixed but making its perimeter to be as large as desired. We thus see that, in
general, it does not seem reasonable to expect a simple characterisation of domains on $\mathbb{S}^{1}\times \R$
satisfying P\'{o}lya's conjecture. Based on the results in the paper, we will, however, conjecture that the area of
a domain may play a role and that for small enough domains the conjecture is sastisfied.
\begin{conjecture}
 There exists a positive number $A_{0}$ such that P\'{o}lya's conjecture is satisfied for all bounded domains
 $\Omega\subset \mathbb{S}^{1}\times \R$ such that $\left|\Omega\right|\leq A_{0}$. 
\end{conjecture}
\noindent A possible quantitative version of the conjecture would take $A_{0}$ to be the upper-end of the lowest
interval in Corollary~\ref{cor-isop}, namely, $A_{0} = 2\pi\left( 8-\sqrt{64-4\pi^2} \right)$.

\section{Precise intervals for the cylinder $\mathbb{S}^{1}\times I_{h}$: proof of Theorem~\ref{mainthm}\label{bounds}}

The Dirichlet problem on $\SI$ may be thought of as an eigenvalue problem on the rectangle $(0,2\pi) \times (0,h)$
with Dirichlet boundary conditions on the sides $(0,2\pi)\times\{0\}$ and $(0,2\pi)\times\{h\}$, and periodic boundary
conditions on the two remaining sides. The corresponding eigenvalues are thus of the form
\begin{equation}\label{eigcyl}
\lambda=m^2+n^2\frac{\pi^2}{h^2}\quad m\in\Z,n\in\N.
\end{equation}
We denote by $N(\lambda)$ the corresponding counting function, that is,
\[
 \begin{array}{lll}
    N(\lambda) & := & \#\left\{\lambda_i : \lambda_i\leq\lambda\right\}\eqskip
               & = & \#\left\{(m, n) \in \Z\times\N : m^2+n^2\fr{\pi^2}{h^2}\leq\lambda\right\}
\end{array}
\]
and want to estimate the order $k$ of the eigenvalue $\lambda_{k}$ as a function of $\lambda$, that
is, to obtain bounds on $k = N(\lambda)$. In this setting, Pólya's conjecture (in dimension two) states that
$$N(\lambda) \leq \lambda\frac{|\SI|}{4\pi}=\lambda\frac{h}{2}.$$


\subsection{Upper bound for Pólya's satisfiability\label{sec-upp}} We begin by extending the argument based on the first eigenvalue
used in~\cite{fs} which was mentioned in the Introduction to see that, in general, not only can P\'{o}lya's conjecture
fail for domains in $\mathbb{S}^{1}\times\R$, but that an inequality of this type cannot hold, even with a smaller constant.

From~\eqref{eigcyl} we have
$$\lambda_1 = \frac{\pi^2}{h^2}$$
and, given any positive constant $C$, by taking $h>\pi^2/(2C)$ we have
\[
 N\left(\lambda_{1}\right) = 1 > \fr{\pi^2}{2C h}= \fr{h}{2C}\lambda_{1}.
\]
We have thus proved the following
\begin{proposition}\label{noweaker}
    Given any positive constant $C$, there exist domains $\Omega$ on the infinite cylinder $\mathbb{S}^{1}\times\R$ for
    which
    \[
     \lambda_{1} < \fr{4C\pi}{|\Omega|}.
    \]
    In particular, P\'{o}lya's conjecture is not satisfied on the cylindrical strip $\ch$ for $h$ larger than $\pi^2/2$.
\end{proposition}

\subsection{Lower bounds for Pólya's satisfiability}
Let
$$R(\lambda) := N(\lambda) - \lambda\frac{h}{2}.$$
P\'{o}lya's conjecture for $\ch$ is then equivalent to $R(\lambda)$ being non-positive.
This can be approximated by counting integer lattice points inside a half-ellipse on the plane using the following construction.
Given an eigenvalue $\lambda$, consider the ellipse centred at the origin defined by the equation
$$e(x, y) := x^2+y^2\frac{\pi^2}{h^2}=\lambda$$
To each integer lattice point $(m, n)\in \Z\times\N$ we associate the eigenvalue
$\lambda_{m,n}:=m^2+n^2\frac{\pi^2}{h^2}$.
In this way $N(\lambda)$ is given by the number of such points lying inside the ellipse 
\[
E = \{(x,y)\in\R\times\R^+_0 : e(x, y)\leq\lambda\},
\]
which are above the horizontal axis.
\begin{figure}[h]
    \centering
    \includegraphics[height=0.25\textheight]{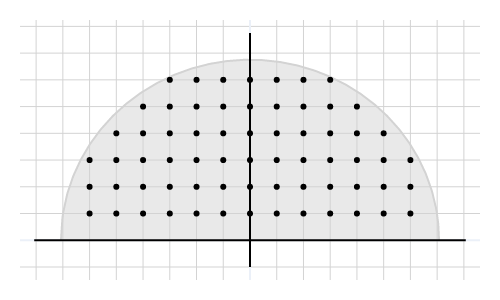}
    \caption{Lattice points and ellipse for $h=3,\lambda=50$}
\end{figure}
We now consider the portion of this region covered by lattice squares given by
$$Q := \bigcup_{\substack{k,l\in\mathds{Z}\\ [k, k+1]\times[l, l+1]\subset E}}[k, k+1]\times[l, l+1].$$
By excluding lattice points lying on the $y$-axis, we can associate each lattice point in $E$ to a
lattice square in $Q$, thus obtaining
\begin{align*}
N(\lambda)&=|Q|+\left\lfloor \frac{h}{\pi}\sqrt{\lambda} \right\rfloor\\
&\leq |E| + \frac{h}{\pi}\sqrt{\lambda} - (|E|-|Q|)\\
&= \lambda\frac{h}{2} + \frac{h}{\pi}\sqrt{\lambda} - |E\setminus Q|
\end{align*}
and
\begin{align*}
R(\lambda)&=N(\lambda)-\lambda\frac{h}{2}\\
&\leq \frac{h}{\pi}\sqrt{\lambda} - |E\setminus Q|.
\end{align*}
As the ellipse is symmetric with respect to the $y$-axis, for area purposes it is enough to consider
the portions of $E$ and $Q$ in the first quadrant
\begin{align*}
E_+ &= E\cap (\R^+_0\times\R^+_0)\\
Q_+ &= Q\cap (\R^+_0\times\R^+_0)
\end{align*}
and now
\begin{equation}
    R(\lambda) \leq \frac{h}{\pi}\sqrt{\lambda} - 2|E_+ \setminus Q_+|\label{rEstimate}
\end{equation}

\subsubsection{First estimate}
Define $\eln := \left\lfloor \frac{h}{\pi}\sqrt{\lambda} \right\rfloor$ and $\epsilon_n := \frac{h}{\pi}\sqrt{\lambda} - \eln$,
and define by
$$
\begin{array}{lll}
e_x(y):= \sqrt{\lambda-y^2\frac{\pi^2}{h^2}} & \mbox{ and } &
e_y(x):= \frac{h}{\pi}\sqrt{\lambda-x^2}
\end{array}
$$
the two functions describing the portion of the ellipse in the first quadrant.

We will approximate the area of $E_+\setminus Q_+$ by partitioning it into the strips

$$S_j := (E_+ \setminus Q_+)\cap(\R\times[j,j+1])\;,\; j=0,1,\dots,\eln$$

Since all these are convex, for $j < \eln$ we can bound the area of each $S_j$ from below by considering the triangle with vertices at $(e_x(j),j)$, $(e_x(j+1),j+1)$, and $(e_x(j+1), j)$.
Similarly, $|S_{\eln}|$ is bounded by the area of the triangle with vertices at $(e_x(\eln),\eln), (0,\eln +\epsilon_n), (0,\eln)$.
\begin{figure}[h]
    \begin{subfigure}{.5\textwidth}
        \centering
        \includegraphics[height=0.25\textheight]{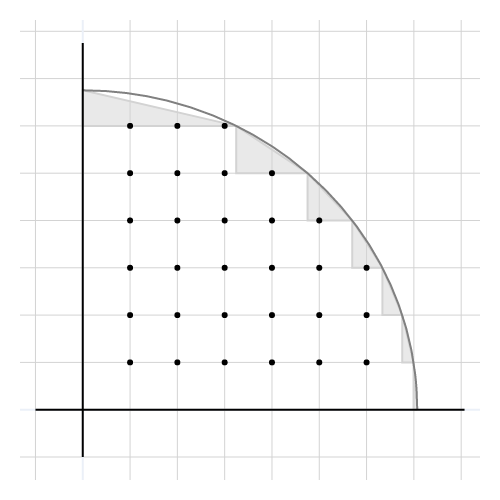}
        \caption{First $R(\lambda)$ estimate for $h=3,\lambda=50$}
    \end{subfigure}%
    \begin{subfigure}{.5\textwidth}
        \centering
        \includegraphics[height=0.25\textheight]{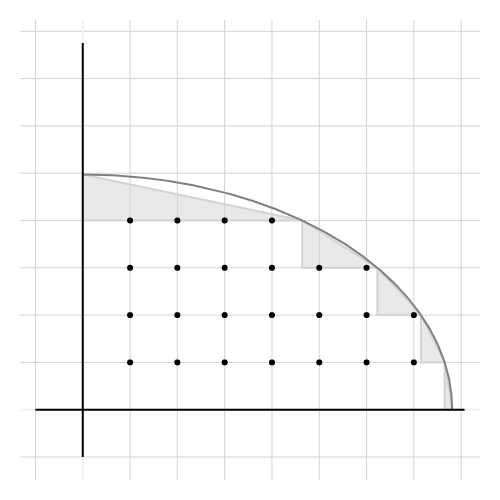}
        \caption{First $R(\lambda)$ estimate for $h=2,\lambda=61$}
    \end{subfigure}
    \caption{First $R(\lambda)$ estimates}
\end{figure}
Thus
\begin{align*}
    \frac{1}{2}|E_+ \setminus Q_+|&\geq \frac{1}{2}e_x(\eln)\epsilon_n + \sum_{j=0}^{\eln-1}\fr{1}{2}(e_x(j)-e_x(j+1))\\
    &=\frac{1}{2}\left[e_x(\eln)\epsilon_n + e_x(0)-e_x(\eln)\right]\\
    &=\frac{1}{2}\left[e_x(\eln)(\epsilon_n-1) + \sqrt{\lambda}\right]
\end{align*}
and
\begin{align*}
    R(\lambda) &\leq \frac{h}{\pi}\sqrt{\lambda} - \left[e_x(\eln)(\epsilon_n-1) + \sqrt{\lambda}\right]\\
    &=\left(\frac{h}{\pi}-1\right)\sqrt{\lambda} - \sqrt{\lambda-\frac{\pi^2}{h^2}\left(\frac{h}{\pi}\sqrt{\lambda}-\epsilon_n\right)^2}(\epsilon_n-1)\\
    &\leq\left(\frac{h}{\pi}-1\right)\sqrt{\lambda} + \sqrt{\frac{\pi^2}{h^2}\left(2\frac{h}{\pi}\sqrt{\lambda}\epsilon_n\right)}(1-\epsilon_n)\\
    &=\left(\frac{h}{\pi}-1\right)\sqrt{\lambda} + \sqrt{\frac{\pi}{h}}\sqrt{2\epsilon_n}(1-\epsilon_n)\lambda^{1/4}.
\end{align*}
Note that $f(x) = \sqrt{2x}(1-x)$ is maximal in $[0,1)$ when $x=\frac{1}{3}$, and thus
\begin{align*}
    R(\lambda)&\leq\left(\frac{h}{\pi}-1\right)\sqrt{\lambda} + \sqrt{\frac{\pi}{h}}\left(\fr{2}{3}\right)^{3/2}\lambda^{1/4}.
\end{align*}
When $h<\pi$ the coefficient of the leading term is negative and
\begin{align*}
    R(\lambda)\leq 0 &\Leftarrow \left(\frac{h}{\pi}-1\right)\sqrt{\lambda} + \sqrt{\frac{\pi}{h}}\left(\fr{2}{3}\right)^{3/2}\lambda^{1/4}\leq 0\\
    &\Leftrightarrow \lambda^{1/4}\geq \frac{\sqrt{\pi/h}\left(\fr{2}{3}\right)^{3/2}}{1-h/\pi}\\
    &\Leftrightarrow \sqrt{\lambda}\geq \frac{8\pi^3}{27h(\pi-h)^2}.
\end{align*}
We thus have
\begin{proposition}
    $R(\lambda) \leq 0$ when $h<\pi$ and $\sqrt{\lambda}\geq \fr{8\pi^3}{27h(\pi-h)^2}$.
\end{proposition}

\subsubsection{Second estimate}
Let $\elm := \left\lfloor \sqrt{\lambda} \right\rfloor$ and $\epsilon_m := \sqrt{\lambda} - \elm$.

Let $(m_0,n_0)$ be the point in the ellipse where its slope is $-1$, that is,
$$e_y'(m_0)=-1\;,\;n_0=e_y(m_0),$$
and let
\begin{align*}
m_i&:= \lfloor m_0\rfloor + i\;,\; i=1,2,\dots,\elm-\lfloor m_0+1\rfloor\\
n_j&:= \lfloor n_0\rfloor + j\;,\; j=1,2,\dots,\elm-\lfloor n_0+1\rfloor.
\end{align*}

Similarly to the first estimate, we partition $E_+\setminus Q_+$ into several parts
\begin{align*}
    R_i &:= (E_+\setminus Q_+)\cap([m_i,m_{i+1}]\times\R)\;,\; i=0,1,2,\dots,\elm-\lfloor m_0\rfloor\\
    S_j &:= (E_+\setminus Q_+)\cap(\R\times[n_j,n_{j+1}])\;,\; j=0,1,2,\dots,\eln-\lfloor n_0\rfloor\\
    T&:= (E_+-Q_+)\cap([0,m_0]\times[0,n_0])
\end{align*}
Since all these are convex, we can estimate their area from below using the following triangles (see Figure~\ref{seconRestimate})
\begin{align*}
    R_i \supset & \text{ Triangle with vertices at }(m_i,e_y(m_i)), (m_{i+1},e_y(m_{i+1})), (m_{i+1},e_y(m_i))\\
    S_j \supset & \text{ Triangle with vertices at }(e_x(n_j),n_j), (e_x(n_{j+1}),n_{j+1}), (e_x(n_{j+1}),n_j)\\
    T \supset & \text{ Triangles with vertices at }(m_0, n_0), (m_0, \min\{\lfloor n_0\rfloor,e_y(m_1)\}), (\lfloor m_0 \rfloor, \min\{\lfloor n_0\rfloor,e_y(m_1)\})\\
    &\text{ or }(m_0, n_0), (\min\{\lfloor m_0\rfloor,e_x(n_1)\}, n_0), (\min\{\lfloor m_0\rfloor,e_x(n_1)\}, \lfloor n_0\rfloor)
\end{align*}
\begin{figure}[h]
    \begin{subfigure}{.5\textwidth}
        \centering
        \includegraphics[height=0.25\textheight]{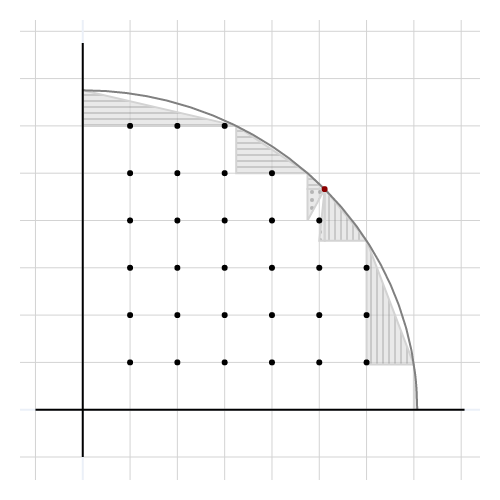}
        \caption{Second $R(\lambda)$ estimate for $h=3,\lambda=50$}
    \end{subfigure}%
    \begin{subfigure}{.5\textwidth}
        \centering
        \includegraphics[height=0.25\textheight]{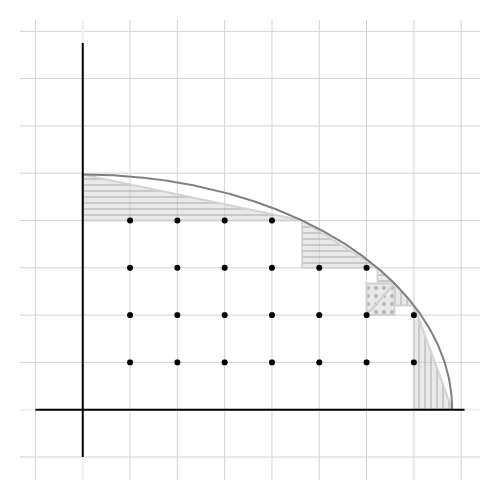}
        \caption{Second $R(\lambda)$ estimate for $h=2,\lambda=61$}
    \end{subfigure}
    \caption{Second $R(\lambda)$ estimates\label{seconRestimate}. $R_i$, $S_j$ and $T$ triangles are shaded with
    vertical lines, horizontal lines, and dots, respectively.}
\end{figure}
to obtain
\begin{align*}
    \left|\bigcup_{i=0}^{\elm-\lfloor m_0 \rfloor} R_i\right|&\geq\sum_{j=0}^{\elm-\lfloor m_0 \rfloor}\frac{1}{2}(e_y(m_i)-e_y(m_{i+1}))(m_{i+1}-m_i)\\
        &= \frac{1}{2}\left[e_y(m_{\elm-\lfloor m_0\rfloor})\epsilon_m + e_y(m_1)-e_y(m_{\elm-\lfloor m_0\rfloor}) + (e_y(m_0)-e_y(m_1))(m_1-m_0)\right]\\
        &= \frac{1}{2}\Big[e_y(\elm)(\epsilon_m-1) + e_y(m_1) + (e_y(m_0)-e_y(m_1))(\lfloor m_0 \rfloor+1-m_0)\Big]\\
        &= \frac{1}{2}\left[\frac{h}{\pi}\sqrt{\lambda - (\sqrt{\lambda}-\epsilon_m)^2}(\epsilon_m-1)
        + e_y(m_0) - (e_y(m_0)-e_y(m_1))(m_0-\lfloor m_0 \rfloor)\right]\\
        &= \frac{1}{2}\left[-\frac{h}{\pi}\sqrt{2\sqrt{\lambda}\epsilon_m - \epsilon_m^2}(1-\epsilon_m)
        + n_0 - (n_0-e_y(m_1))(m_0-\lfloor m_0 \rfloor)\right]\\
        &\geq \frac{1}{2}\left[-\frac{h}{\pi}\sqrt{2\epsilon_m}(1-\epsilon_m)\lambda^{1/4}
        + n_0 - (n_0-e_y(m_1))(m_0-\lfloor m_0 \rfloor)\right].
\end{align*}
We again use the fact that $f(x)=\sqrt{2x}(1-x)$ attains its maximum in $[0,1)$ when $x=\frac{1}{3}$, and thus
\begin{align*}        
    \left|\bigcup_{i=0}^{\elm-\lfloor m_0 \rfloor} R_i\right|&\geq \frac{1}{2}\left[-\frac{h}{\pi}\left(\fr{2}{3}\right)^{3/2}\lambda^{1/4} + n_0 - (n_0-e_y(m_1))(m_0-\lfloor m_0 \rfloor)\right].
\end{align*}
Similarly
$$\left|\bigcup_{j=0}^{\eln-\lfloor n_0 \rfloor} S_j\right|\geq
    \frac{1}{2}\left[-\sqrt{\frac{\pi}{h}}\left(\fr{2}{3}\right)^{3/2}\lambda^{1/4}+ m_0 - (m_0-e_x(n_1))(n_0-\lfloor n_0 \rfloor)\right],
    $$
and
$$|T|\geq \frac{1}{2}\left[(m_0-\lfloor m_0 \rfloor)(n_0-e_y(m_1))+(n_0-\lfloor n_0 \rfloor)(m_0-e_x(n_1))\right].$$
Finally
\begin{alignat*}{3}
    \frac{1}{2}|E_+\setminus Q_+|&\geq \frac{1}{2}\bigg[&&-\frac{h}{\pi}\left(\fr{2}{3}\right)^{3/2}\lambda^{1/4} + n_0 - (n_0-e_y(m_1))(m_0-\lfloor m_0 \rfloor)\\
        & & &-\sqrt{\frac{\pi}{h}}\left(\fr{2}{3}\right)^{3/2}\lambda^{1/4}+ m_0 - (m_0-e_x(n_1))(n_0-\lfloor n_0 \rfloor)\\
        & & & \hspace*{5mm}+\left(m_0-\lfloor m_0 \rfloor\right)(n_0-e_y(m_1))+(n_0-\lfloor n_0 \rfloor)(m_0-e_x(n_1))\bigg]\\
    &=\frac{1}{2}\Big[&&-\left(\frac{h}{\pi}+\sqrt{\frac{\pi}{h}}\right)\left(\fr{2}{3}\right)^{3/2}\lambda^{1/4} + m_0 + n_0 \Big].
\end{alignat*}
We can now determine the points on the ellipse with slope $-1$, namely,
$$e_y'(m_0)=-1\Leftrightarrow \frac{h m_0}{\pi\sqrt{\lambda-m_0}} = 1 \Leftrightarrow m_0 = \frac{\sqrt{\lambda}}{\sqrt{(h/\pi)^2+1}}$$
$$n_0=e_y(m_0)=\frac{h}{\pi}\sqrt{\lambda-\frac{\lambda}{(h/\pi)^2+1}} = \frac{(h/\pi)^2\sqrt{\lambda}}{\sqrt{(h/\pi)^2+1}},$$
and substitute in the expression above to obtain
$$
\frac{1}{2}|E_+\setminus Q_+|\geq \frac{1}{2}\bigg[-\left(\frac{h}{\pi}+\sqrt{\frac{\pi}{h}}\right)\left(\fr{2}{3}\right)^{3/2}\lambda^{1/4} + \frac{\sqrt{\lambda}}{\sqrt{(h/\pi)^2+1}} + \frac{(h/\pi)^2\sqrt{\lambda}}{\sqrt{(h/\pi)^2+1}} \bigg].
$$
From~\eqref{rEstimate} we have
\begin{align}
    R(\lambda)&\leq\frac{h}{\pi}\sqrt{\lambda}+\left(\frac{h}{\pi}+\sqrt{\frac{\pi}{h}}\right)\left(\fr{2}{3}\right)^{3/2}\lambda^{1/4}-\frac{((h/\pi)^2+1)\sqrt{\lambda}}{\sqrt{(h/\pi)^2+1}}\nonumber \\
        &=\left(\fr{h}{\pi}-\sqrt{\left(\fr{h}{\pi}\right)^2+1}\right)\sqrt{\lambda}+\left(\frac{h}{\pi}+\sqrt{\frac{\pi}{h}}\right)\left(\fr{2}{3}\right)^{3/2}\lambda^{1/4}\nonumber \\
        &=-\frac{\sqrt{h^2+\pi^2}-h}{\pi}\sqrt{\lambda}+\left(\frac{h}{\pi}+\sqrt{\frac{\pi}{h}}\right)
        \left(\fr{2}{3}\right)^{3/2}\lambda^{1/4}\label{polya-estim2}. \\
\end{align}
But
\begin{align*}
0&\geq-\frac{\sqrt{h^2+\pi^2}-h}{\pi}\sqrt{\lambda}+\left(\frac{h}{\pi}+\sqrt{\frac{\pi}{h}}\right)\left(\fr{2}{3}\right)^{3/2}\lambda^{1/4}\\
\Leftrightarrow \lambda^{1/4}&\geq \left(\frac{h}{\pi}+\sqrt{\frac{\pi}{h}}\right)\left(\fr{2}{3}\right)^{3/2}\frac{\pi}{\sqrt{h^2+\pi^2}-h}\\
&=\left(\fr{2}{3}\right)^{3/2}\frac{h+\pi^{3/2}\sqrt{1/h}}{\sqrt{h^2+\pi^2}-h}\\
\Leftrightarrow \sqrt{\lambda}&\geq\fr{8}{27}\frac{(h+\pi^{3/2}\sqrt{1/h})^2}{(\sqrt{h^2+\pi^2}-h)^2}\\
&=\fr{8}{27}\frac{(h/\pi+\sqrt{\pi/h})^2}{(\sqrt{(h/\pi)^2+1}-h/\pi)^2}
\end{align*}
and we have thus proven the following bound.
\begin{proposition}\label{secestim}
    $R(\lambda) \leq 0$ when $\sqrt{\lambda}\geq \fr{8}{27}\left(\fr{h/\pi+\sqrt{\pi/h}}{\sqrt{(h/\pi)^2+1}-h/\pi}\right)^2$.
\end{proposition}

\subsubsection{Lower bounds}
We have
$$\lambda_1 = \frac{\pi^2}{h^2} \Rightarrow \sqrt{\lambda}\geq\frac{\pi}{h}$$
and, for $h<\pi$,
$$\frac{\pi}{h} \geq\frac{8\pi^3}{27h(\pi-h)^2}\Leftrightarrow
(\pi-h)^2 \geq\frac{8}{27}\pi^2\Leftrightarrow
h\leq \frac{1}{9}(9-2\sqrt{6})\pi<\pi.$$
Thus
$$\sqrt{\lambda} \geq \frac{8\pi^3}{27h(\pi-h)^2} \text{ when } h\leq \frac{1}{9}(9-2\sqrt{6})\pi,$$
and finally
$$R(\lambda) \leq 0 \text{ when } h\leq \frac{1}{9}(9-2\sqrt{6})\pi.$$
We have thus
\begin{proposition}
    $\SI$ satisfies Pólya's conjecture for $h\leq \left(1-\fr{2\sqrt{6}}{9}\right)\pi$.
\end{proposition}
For the remaining values $\frac{1}{9}(9-2\sqrt{6})\pi < h \leq \frac{\pi^2}{2}$ we first note that the
second estimate given in Proposition~\ref{secestim} is convex in $h$ and bounded above by
\begin{alignat*}{3}
    \left(\fr{2}{3}\right)^3\frac{(h/\pi+\sqrt{\pi/h})^2}{(\sqrt{(h/\pi)^2+1}-h/\pi)^2}
    &\leq \max\Bigg\{
    &&\left(\fr{2}{3}\right)^3\left.\frac{(h/\pi+\sqrt{\pi/h})^2}{(\sqrt{(h/\pi)^2+1}-h/\pi)^2}\right|_{h=\frac{1}
    {9}(9-2\sqrt{6})\pi},\\
    & &&\left(\fr{2}{3}\right)^3\left.\frac{(h/\pi+\sqrt{\pi/h})^2}{(\sqrt{(h/\pi)^2+1}-h/\pi)^2}\right|
    _{h=\frac{\pi^2}{2}}
    \Bigg\}\\
    &<20,&&
\end{alignat*}
yielding $R(\lambda) < 0$  whenever $\sqrt{\lambda}\geq 20.$
\begin{proposition}
    $\SI$  satisfies Pólya's conjecture for $\frac{1}{9}(9-2\sqrt{6})\pi < h \leq \frac{\pi^2}{2}$ and $\lambda \geq 400$
\end{proposition}
\subsection{Computational search} \label{subsec:polya/cs}
To determine the satisfiability when $\frac{1}{9}(9-2\sqrt{6})\pi < h \leq \frac{\pi^2}{2}$ and
$$\lambda < 400 \Rightarrow k = N(\lambda) \leq \frac{h}{2}\lambda +  \frac{h}{\pi}\sqrt{\lambda} = 400\frac{h}{2} + 20\frac{h}{\pi} < 1019$$
we first discretize the problem to reduce it to a finite number of cases, and then use a computer-aided search
to check that the conjecture holds in each of the resulting cases.
\subsubsection{$k$-intervals} Consider again the expression for the eigenvalues of $\ch$ given by~\eqref{eigcyl},
and define the $k$-intervals
\begin{align*}
I_{m,n}^k:&=\left\{h\in\mathds{R}^+ : k > \lambda_{m,n}\frac{h}{2}\right\}\\
&=\left\{h\in\mathds{R}^+ : k > m^2\frac{h}{2} + n^2\frac{\pi^2}{2h}\right\}.
\end{align*}
Intuitively, the $k$-interval $I_{m,n}^k$ is the range of values of $h$ for which the eigenvalue $\lambda_{m,n}$ would not satisfy Pólya's conjecture if it were the $k$-th eigenvalue.
It is guaranteed that these sets are indeed intervals since $m^2\frac{h}{2} + n^2\frac{\pi^2}{2h}$ is convex for $h\in(0, \infty)$.

Now we provide a necessary and sufficient condition for P\'{o}lya's conjecture to fail for a particular eigenvalue $\lambda_{k}$.
\begin{lemma}
For any fixed value of $h$, P\'{o}lya's conjecture fails for the $k$-th eigenvalue $\lambda_{k}$
if and only if $h$ is in at least $k$ $k$-intervals. More precisely, $\lambda_k < \frac{2k}{h}$ if and
only if $$h\in I_{m_1,n_1}^k\cap I_{m_2,n_2}^k\cap\dots\cap I_{m_k,n_k}^k$$
where $\lambda_{m_1,n_1}, \lambda_{m_2,n_2}, \dots, \lambda_{m_k,n_k}$ are the first $k$ eigenvalues.
\begin{proof}
    First note that for a fixed $h$, and if $(m,n)$ and $(m',n')$ are such that $\lambda_{m,n} < \lambda_{m',n'}$,
    if $h$ is in $I_{m',n'}^k$, then it must also belong to $I_{m,n}^k$, as
    $$  \lambda_{m,n}\frac{h}{2} < \lambda_{m',n'}\frac{h}{2} < k.$$
    If $k>\lambda_k\frac{h}{2}$, we know that $h$ is in the $k$-interval corresponding to $\lambda_k$. Using the previous observation we deduce that $h$ is also in all the $k$-intervals corresponding to $\lambda_1,\dots,\lambda_{k-1}$, thus, $h$ is in at least $k$ $k$-intervals.
    
    If $h$ is in at least $k$ $k$-intervals then at least one of them corresponds to an eigenvalue $\lambda_l,\; l\geq k$ and we can again use the previous result to conclude that $h$ is in all the $k$-intervals corresponding to any eigenvalue smaller than $\lambda_l$, specifically $\lambda_l > \lambda_k$ and thus $k>\lambda_k\frac{h}{2}$.
\end{proof}
\end{lemma}
\subsubsection{Search}
$\lambda_{m,n}$ is increasing over $\left|m\right|$ and $n$ so if $\lambda_{m,n}$ is the $k$-th smallest eigenvalue, we have
$$|m|, n \leq k$$
We only need to consider $k\leq 1019$ so it suffices to
\begin{enumerate}[i.]
    \item compute $I^k_{m,n}$ for $|m|, n, k\leq 1\,019$;
    \item determine all intersections of at least $k$ $k$-intervals for each $k=1,2,\dots,1\,019$
\end{enumerate}
to determine the satisfiability of Pólya's conjecture for $\SI$ in the remaining cases.

We can calculate this algebraically as we can determine exact expressions for every $k-$interval's endpoints. Implementing and running this search (details are included in Appendix~\ref{apx:impl}) gives us the intersections
\[
\begin{array}{ll}
 h\in\left(-\sqrt{-4\pi^2+64}+8,\sqrt{-\frac{\pi^2}{4}+4}+2\right) & \mbox{ for } k=8\eqskip
 h\in\left(-\sqrt{-9\pi^2+169}+13,\sqrt{-\pi^2+\frac{169}{16}}+\frac{13}{4}\right) & \mbox{ for } k=13.
\end{array}
\]


Combining all the results above yields a proof of Theorem~\ref{mainthm}.

\begin{figure}[h]
    \begin{subfigure}{0.5\textwidth}
        \centering
        \includegraphics[width=\textwidth]{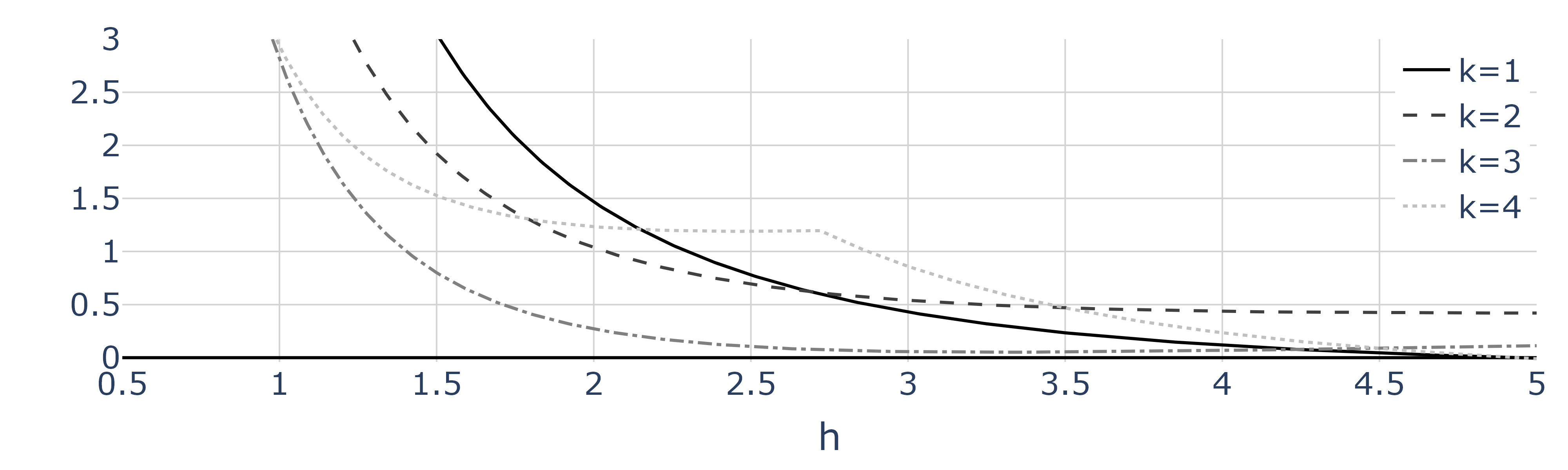}
        \caption{$k\in[1,5)$}
    \end{subfigure}%
    \begin{subfigure}{0.5\textwidth}
        \centering
        \includegraphics[width=\textwidth]{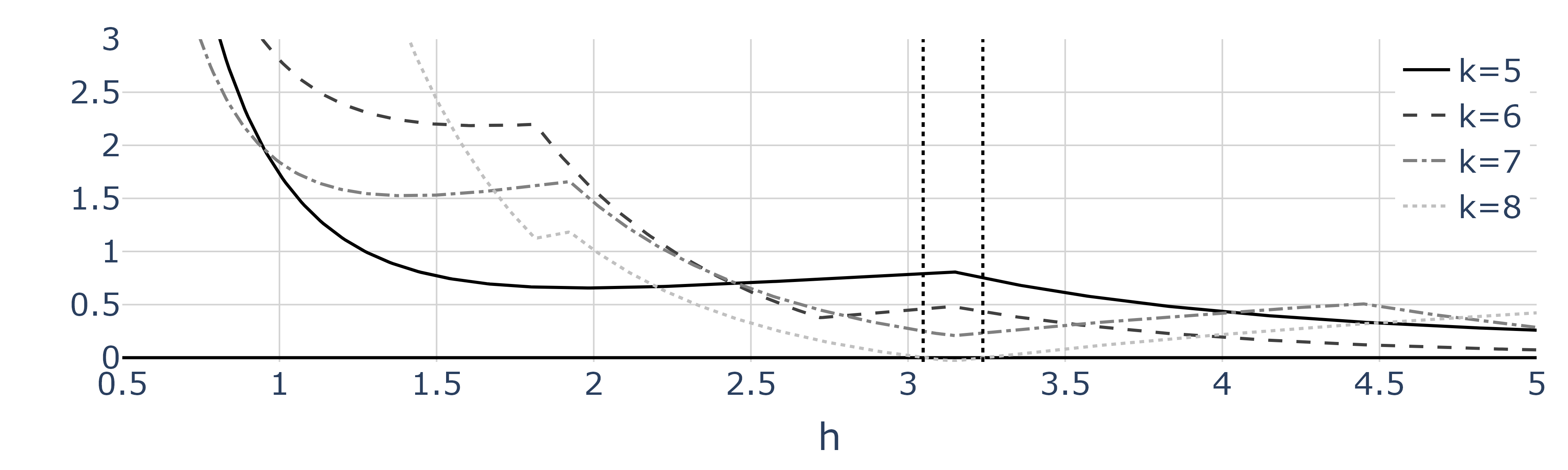}
        \caption{$k\in[5,10)$}
    \end{subfigure}

    \begin{subfigure}{0.5\textwidth}
        \centering
        \includegraphics[width=\textwidth]{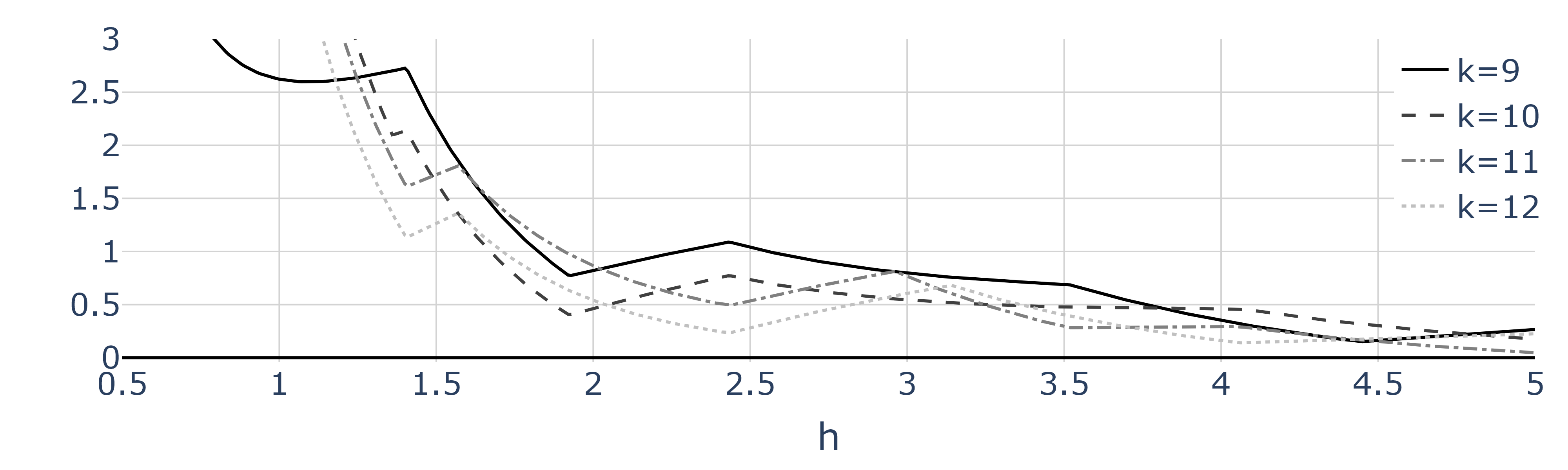}
        \caption{$k\in[10,15)$}
    \end{subfigure}%
    \begin{subfigure}{0.5\textwidth}
        \centering
        \includegraphics[width=\textwidth]{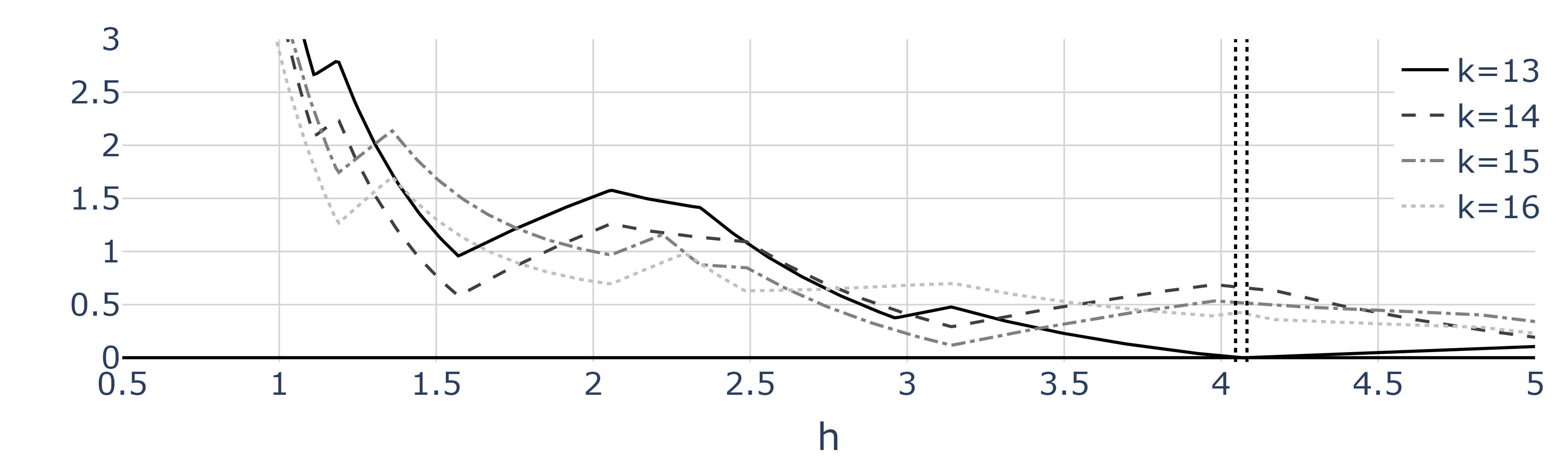}
        \caption{$k\in[10,15)$}
    \end{subfigure}

    \caption{Plots of $(\lambda_k-2k/h)/\sqrt{k}$ for different values of $k$; the two pairs of vertical
    lines denote the intervals where $\lambda_{8}$ and $\lambda_{13}$ fail the conjecture and which are
    shown in Figure~\ref{nesxtfig}.}
\end{figure}

\begin{figure}[h]
    \begin{subfigure}{.5\textwidth}
        \centering
        \includegraphics[width=\textwidth]{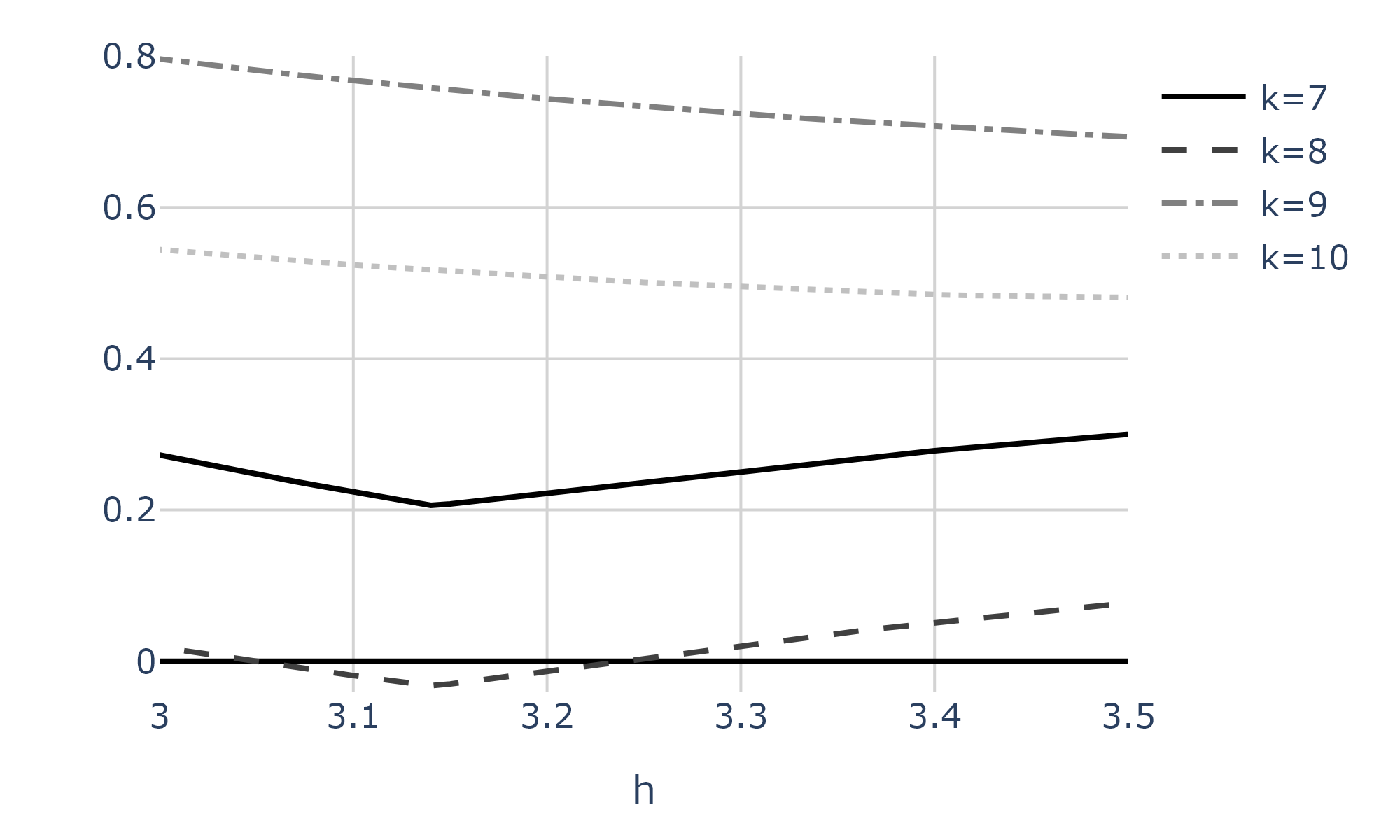}
        \caption{$\lambda_{8}$}
    \end{subfigure}%
    \begin{subfigure}{.5\textwidth}
        \centering
        \includegraphics[width=\textwidth]{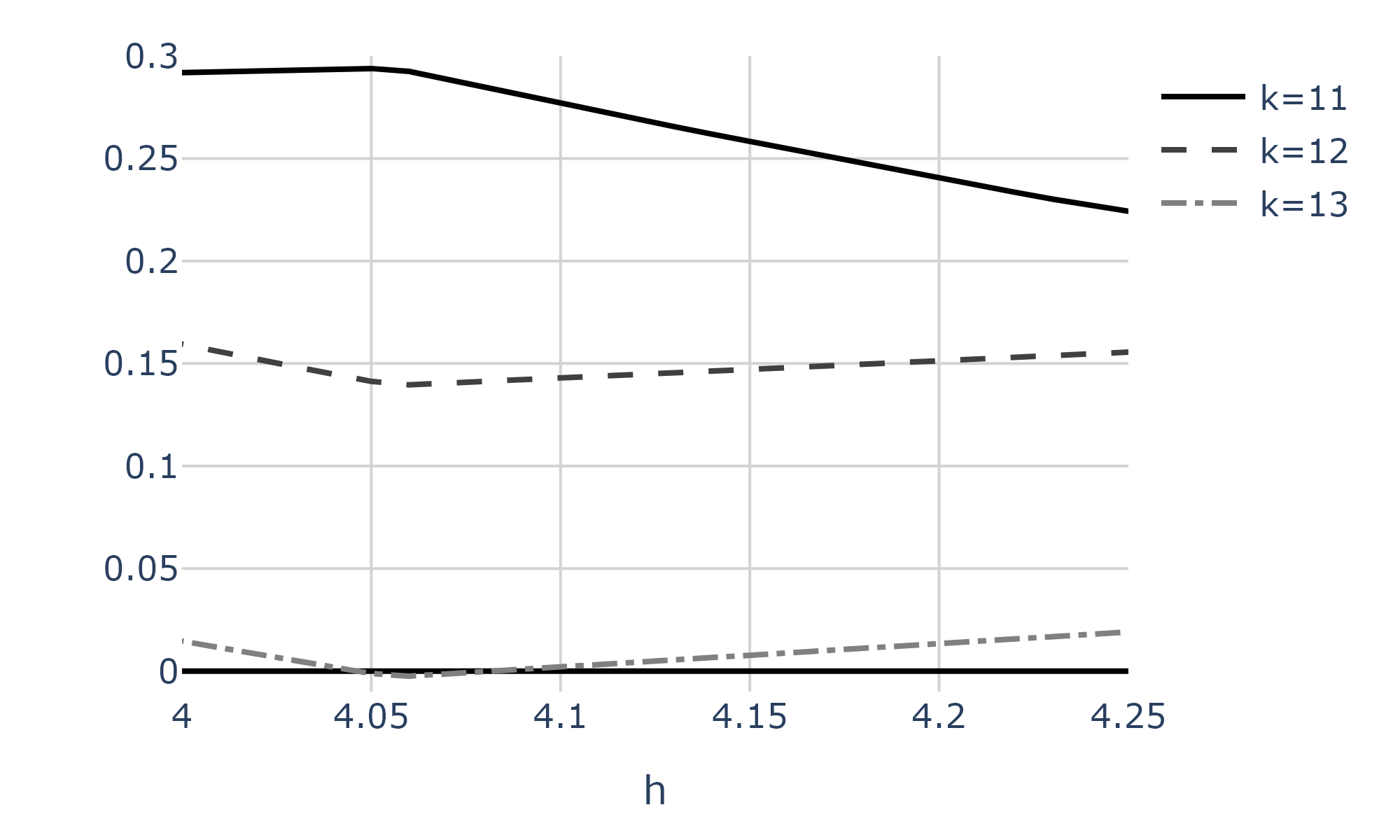}
        \caption{$\lambda_{13}$}
    \end{subfigure}%

    \caption{Plots of $(\lambda_k-2k/h)/\sqrt{k}$ showing the two intervals where Pólya's conjecture is
    not satisfied.\label{nesxtfig}}
\end{figure}

\section{Li-Yau inequalities\label{liyau-sec}}

The Li-Yau inequalities for $\SI$ are given by~\eqref{liyau}.
Proceeding as in the case of the first eigenvalue for P\'{o}lya's inequality, we now see that for $h$ larger than
$\pi^2$ this cannot hold already for $k$ equal to one.
We will, however, show that~\eqref{liyau} holds for all $h\in(0,\pi^2]$. To do this, it is convenient to prove the result
separately on $(0,\pi^2/2]$ and then on $(\pi^2/2,\pi^2]$, and we will start with the former interval. We make use of the following two lemmas.
\begin{lemma}
    \label{lemma:liyau1}
    If the Li-Yau inequality~\eqref{liyau} holds up to $k-1$ and $\lambda_{k}$ satisfies Pólya's conjecture, then
    it also holds for $k$.
\end{lemma}
\begin{proof}
    Follows directly by induction.
\end{proof}

Using this lemma, and for the first range of $h$ under consideration, we only need to show that the inequality holds when
P\'{o}lya's conjecture is not satisfied (for $\lambda_8$ and $\lambda_{13}$) and then everything else will follow by
induction. To do this, we note that the following, slightly weaker version of Lemma~\ref{lemma:liyau1}, also holds.
\begin{lemma}
    \label{lemma:liyau2}
    If $\lambda_1,\lambda_2,\dots,\lambda_{k-1}$ all satisfy Pólya's conjecture and
    $$\lambda_{k}\geq\frac{k}{h},$$
    then the Li-Yau inequality holds for $\lambda_1,\lambda_2,\dots,\lambda_{k}$.
    \begin{proof}
        $$\sum_{i=1}^k \lambda_i \geq \frac{2(k-1)k}{h}+\frac{k}{h}=\frac{(2k-1)k}{h}\geq\frac{k^2}{h}$$
    \end{proof}
\end{lemma}
Thus, if the weaker version $\lambda_{k}\geq k/h$ holds for the only eigenvalues for which P\'{o}ly\'{a}'s conjecture
is not satisfied in this range, namely $\lambda_8$ and $\lambda_{13}$, we will have shown that~\eqref{liyau} holds on the first
interval.
It is straightforward to derive the precise values of $\lambda_{8}$ and $\lambda_{13}$ in the non-satisfiability intervals via direct
calculations yielding
\begin{align*}
    \lambda_8 &= \begin{cases}
        1 + 4 \frac{\pi^2}{h^2}\quad &h\in\big(8 - \sqrt{64 - 4 \pi^2}, \pi\big]\eqskip
        4 + \frac{\pi^2}{h^2} \quad &h\in\big[\pi, 2 + \sqrt{4 - \pi^2/4}\big)
    \end{cases}
\end{align*}
and
\begin{align*}
    \lambda_{13} &= \begin{cases}
        1 +  9 \frac{\pi^2}{h^2}\quad &h\in\big(13 - \sqrt{169 - 9 \pi^2}, \sqrt{5/3} \pi\big]\eqskip
        4 +  4 \frac{\pi^2}{h^2}\quad &h\in\big[\sqrt{5/3} \pi, 13/4 + \sqrt{169/16 - \pi^2}\big).
    \end{cases}.
\end{align*}
Since
\begin{align*}
    1 + 4 \frac{\pi^2}{h^2} \geq\frac{8}{h} &\Leftrightarrow (h-4)^2\geq 16-4\pi^2 \,(< 0)\\
    4 + \frac{\pi^2}{h^2} \geq\frac{8}{h} &\Leftrightarrow (h-1)^2\geq 1-\frac{\pi^2}{4} \,(< 0)\\
    1 +  9 \frac{\pi^2}{h^2} \geq\frac{13}{h} &\Leftrightarrow \left(h-\frac{13}{2}\right)^2\geq \frac{169}{4} - 9\pi^2 \,(< 0)\\
    4 +  4 \frac{\pi^2}{h^2} \geq\frac{13}{h} &\Leftrightarrow \left(h-\frac{13}{8}\right)^2\geq \frac{169}{64} - \pi^2 \,(< 0)\\
\end{align*}
we conclude that
$\lambda_8\geq\fr{8}{h}$ and $\lambda_{13}\geq\fr{13}{h}$ for $h\in(0,\pi^2/2]$.
Combining this with the previous two lemmas we obtain the following result.
\begin{proposition}
    The Li-Yau inequalities on $\SI$ are satisfied if $h$ belongs to
    $$\left(0,\frac{\pi^2}{2}\right]$$
\end{proposition}

To prove Li-Yau's satisfiability when $h$ belongs to $(\pi^2/2,\pi^2]$ we need a slightly stronger lemma
\begin{lemma}
    \label{lemma:liyau3}
    If the Li-Yau inequality~\eqref{liyau} holds up to $k-1$ and
    $$\lambda_{k}\geq\frac{2k-1}{h},$$
    then it also holds for $k$.
\end{lemma}
\begin{proof}
    Follows directly by induction.
\end{proof}

If we can guarantee Li-Yau's satisfiability where $\lambda_{k}<(2k-1)/h$ then the all other inequalities will follow by induction from this lemma.

From estimative~\eqref{polya-estim2} we obtain the bound
\begin{align*}
    R(\lambda)=N(\lambda)-\lambda\frac{h}{2}&\leq-\frac{\sqrt{h^2+\pi^2}-h}{\pi}\sqrt{\lambda}+\left(\frac{h}{\pi}+\sqrt{\frac{\pi}{h}}\right)\left(\fr{2}{3}\right)^{3/2}\lambda^{1/4}\\
    \Leftrightarrow k&\leq \lambda\frac{h}{2}-\frac{\sqrt{h^2+\pi^2}-h}{\pi}\sqrt{\lambda}+\left(\frac{h}{\pi}+\sqrt{\frac{\pi}{h}}\right)\left(\fr{2}{3}\right)^{3/2}\lambda^{1/4}
\end{align*}
Now if
$$-\frac{\sqrt{h^2+\pi^2}-h}{\pi}\sqrt{\lambda}+\left(\frac{h}{\pi}+\sqrt{\frac{\pi}{h}}\right)\left(\fr{2}{3}\right)^{3/2}\lambda^{1/4}\leq\frac{1}{2},$$
then the lemma's condition is satisfied as
$$k\leq \lambda\frac{h}{2}+\frac{1}{2}\Leftrightarrow \lambda\geq\frac{2k-1}{h}.$$
It is thus enough to check that Li-Yau's inequalities hold for all values of $\lambda$ such that
\begin{align*}
    \frac{1}{2}&<-\frac{\sqrt{h^2+\pi^2}-h}{\pi}\sqrt{\lambda}+\left(\frac{h}{\pi}+\sqrt{\frac{\pi}{h}}\right)\left(\fr{2}{3}\right)^{3/2}\lambda^{1/4}\\
    \Rightarrow \frac{1}{2}&<-\frac{\sqrt{\pi^4+\pi^2}-\pi^2}{\pi}\sqrt{\lambda}+\left(\frac{\pi^2}{\pi}+\sqrt{\frac{\pi}{\pi^2}}\right)\left(\fr{2}{3}\right)^{3/2}\lambda^{1/4}
\end{align*}
since the right--hand side is increasing in $h$ on the interval $(\pi^2/2,\pi^2]$. This is then equivalent to
\[
\frac{1}{2}<-(\liyauA)\sqrt{\lambda}+\liyauB\lambda^{1/4},
\]
which may be solved to yield
\begin{gather*}
    \fr{2\pi^{3/2}+2-\sqrt {-27\pi\sqrt {1 + \pi^2} + 8\pi^{3/2} + 4\pi^3 + 27\pi^2 + 4}} {3\sqrt {6\pi}\left (\sqrt {1 + \pi^2} - \pi  \right)}
    < \lambda^{1/4} < \fr{2\pi^{3/2} + 2+\sqrt {-27\pi\sqrt {1 + \pi^2} + 8\pi^{3/2} + 4\pi^3 + 27\pi^2 + 4}} {3\sqrt {6\pi}\left (\sqrt {1 + \pi^2} - \pi  \right)}.
\end{gather*}
The lower and upper ends of the above interval correspond to $\lambda\approx 0.004$ and
$\lambda\approx 26\,300.6$, respectively. We now determine which eigenvalues in this range do not satisfy the
condition in Lemma~\ref{lemma:liyau3}, namely,
$$\lambda\geq\frac{2k-1}{h}\Leftrightarrow k-\fr{1}{2}\leq\lambda\frac{h}{2}.$$
As this condition is of the same type as that for Pólya's conjecture, we can use the same algorithm as
in Section~\ref{bounds} to computationally check these remaining cases, with the modified $k$-intervals

$$\tilde{I}_{m,n}^k:=\left\{h\in\mathds{R}^+ : k-\fr{1}{2} > m^2\frac{h}{2} + n^2\frac{\pi^2}{2h}\right\}.$$

Since
$$k\leq\frac{h}{2}\lambda+\frac{h}{\pi}\sqrt{\lambda}<130\,298$$
we only need to consider the first $130\,297$ eigenvalues. To compute this we need the following improvements to the previous algorithm:

\begin{enumerate}[i.]
    \item Calculating the $k$-intervals, we only need to consider eigenvalues $\lambda=m^2+n^2(\pi^2/h^2)$ which can be smaller than $26\,300.6$ in the range $h\in(\pi^2/2,\pi^2]$ since we are only interested in intersecting $k$-intervals if $\lambda_k\lesssim26\,300.65 \Rightarrow \lambda_1,\dots,\lambda_k\lesssim26\,300.6$. This means we can skip values of $m,n$ when
    $$m^2+n^2\frac{1}{\pi^2}>26\,301(>26\,300.6)$$
    \item Since $\lambda=m^2+n^2(\pi^2/h^2)$ and $\lambda\lesssim26\,300.6$, we only ever need to consider values for $m, n$ in the ranges
    \begin{align*}
    |m|&\leq\sqrt{\lambda}<163\\
    n&\leq(h/\pi)\sqrt{\lambda}\leq\pi\sqrt{\lambda}<510
    \end{align*}
\end{enumerate}

Running this version of the algorithm (details are given in Appedix~\ref{apx:impl2}) we obtain that the only eigenvalue orders which do not satisfy the condition in Lemma~\ref{lemma:liyau1} in the range
$h\in(\pi^2/2,\pi^2], \lambda < 26\, 301$ are $7,10,77$, and $86$.
This means that we only need to check that the Li-Yau's inequalities hold up to $k=86$. For any fixed value of $h$ we can easily compute the exact values of $\lambda_1,\dots,\lambda_{86}$. In fact, for a range of values $h\in[h_0,h_1]$ if there are no eigenvalue ``crossovers'' of the form

$$\lambda_i = m_i^2 + n_i^2\frac{\pi^2}{h^2}=m_j^2 + n_j^2\frac{\pi^2}{h^2} = \lambda_j$$
for some $h\in(h_0,h_1)$ with $i,j\leq86$ and $(|m_1|,n_1)\neq(|m_2|,n_2)$ then we can determine the exact expressions of $\lambda_1,\dots,\lambda_{86}$ as functions of $h$ since the order of these expressions never changes; to do this we can choose the value $h_M=(h_0+h_1)/2$, consider each expression $m^2 + n^2\frac{\pi^2}{h^2}$ evaluated at $h=h_M$ and order them based on their exact values.

To partition $(\pi^2/2, \pi^2]$ into intervals without ``crossovers'' we may compute all the intersections between eigenvalues as follows: for each $(|m_1|,n_1)\neq(|m_2|,n_2)$ such that $N(\lambda_{m_1,n_1}),N(\lambda_{m_2,n_2})\leq 86$ we add the point
\[
\begin{array}{lll}
    \lambda_{m_1,n_1} = \lambda_{m_2,n_2} & \Leftrightarrow & m_1^2 + n_1^2\fr{\pi^2}{h^2}=m_2^2 + n_2^2\frac{\pi^2}{h^2}\eqskip
    & \Leftrightarrow & h = \pi\sqrt{\frac{n_2^2-n_1^2}{m_1^2-m_2^2}}
\end{array}
\]
to a list and after we sort it to form the partition.

On each interval $[h_0,h_1]$ of the partition we can now compute the exact expressions of the sums $\lambda_1,\lambda_1+\lambda_2,\dots,\sum_{i=1}^{86}\lambda_i$ which will be of the form

$$\sum_{i=1}^{k}\lambda_i = \sum_{i=1}^{k}m_i^2+n_i^2\frac{\pi^2}{h^2}=a_k+b_k\frac{1}{h^2}$$
for some constants $a_k, b_k$. Finally, we determine for which values of $h$ the Li-Yau inequality holds
\[
\begin{array}{lll}
    \dsum_{i=1}^{k}\lambda_i=a_k+b_k\frac{1}{h^2}\geq\frac{k^2}{h} & \Leftrightarrow & a_k h^2-k^2 h + b_k\geq0 \eqskip
    & \Leftrightarrow & h \in
\begin{cases}
    \left(-\infty, \frac{b_k}{k^2}\right] &\text{if $a_k=0$}\eqskip
    \left(-\infty, \frac{k^2-\sqrt{k^4-4a_kb_k}}{2a_k}\right] \cup \left[\frac{k^2+\sqrt{k^4-4a_kb_k}}{2a_k}, \infty\right) &\text{if $a_k\neq0$}\eqskip
\end{cases}
\end{array}
\]
and check if this region covers all the interval $[h_0, h_1]$.

The algorithm described in Appendix \ref{apx:impl3} implements all the verification steps and running it confirms that indeed the Li-Yau inequalities are satisfied up to $k=86$ when $h\in(\pi^2/2,\pi^2]$ and $\lambda < 26\,301$. Since the condition in Lemma~\ref{lemma:liyau3} holds for all $k>86$ when $h\in(\pi^2/2,\pi^2]$ we conclude that Li-Yau inequalities are satisfied for any $k$ when $h\in(\pi^2/2,\pi^2]$ as intended, thus finishing
the proof of Corollary~\ref{cor-liyau}.

\section{Remarks for other domains\label{other}}

\subsection{Domains on the cylindrical surface $\mathbb{S}^{1}\times \R$}

If a domain $\Omega\subset \mathbb{S}^{1}\times \R$ contains a cylindrical strip $\ch$ for some value of $h$,
we then have
$
 \lambda_{k}\left(\Omega\right) \leq \lambda_{k}\left(\ch\right)
$
and, in particular, $\lambda_{1}\left(\Omega\right) \leq \lambda_{1}\left(\ch\right)=\pi^2/h^2$. Hence, if $h$
is such that
\[
\fr{\pi^2}{h^2} \leq \fr{4\pi}{\left|\Omega\right|},
\]
it follows that $\Omega$ does not satisfy P\'{o}lya's conjecture. On the other hand, one class of domains for
which the conjecture will be satisfied will be those tiling a strip $\ch$ which itself satisfies the conjecture
-- see~\cite{fms} for the corresponding argument, for instance.

We shall now specifically consider the geodesic disk $D_{R}$ of radius $R$ defined on the cylindrical surface $\mathbb{S}^1\times \R$, for which it is known that it satisfies P\'{o}lya's conjecture for $R$ up to $
\pi$~\cite{flps}, and provide an upper bound for the value of $R$ where this stops being the case. Note that
since the area of the domain $\Omega$ in the expression above also depends on $h$, to apply this type
of argument it is necessary to solve an equation in $h$. However, the resulting value of $R$ is approximately $6.85504$ and we will use another method to derive a better estimate.
\begin{thm}
There exists a value of $R$, say $R_{1}$, such that the geodesic disk $D_{R}$ on the cylinder $\mathbb{S}^{1}\times \R$
does not satisfy P\'{o}lya's conjecture for $R>R_{1}$. The constant $R_{1}$ may be taken to be $3.76085$.
\end{thm}
\begin{proof} We use the function
\[
 u(x,y) = \cos\left (\fr {\pi   y} {2\sqrt {R^2 - x^2}} \right)
\]
as a test function for the corresponding Rayleigh quotient to obtain an upper bound for the first eigenvalue on $D_{R}$.
For $R$ larger than or equal to $\pi$, we have
\[
\begin{array}{lll}
 \dint_{D_{R}} u^{2}(x,y)\, {\rm d}y{\rm d}x & = & \pi\sqrt {R^2 - \pi^2} +
   R^2\arctan\left (\fr{\pi} {\sqrt {R^2 - \pi^2}} \right)\eqskip
\end{array}
\]
while
\[
\begin{array}{lll}
 \dint_{D_{R}} \left|\nabla u\right|^{2}(x,y)\, {\rm d}y{\rm d}x & = &
 \fr {1} {6}\left[\fr {\pi\left (6 + \pi^2 \right)} {\sqrt {R^2 - \pi^2}} + 2\left (\pi^2 - 3 \right)
      \arctan\left (\fr {\pi} {\sqrt {R^2 - \pi^2}} \right) \right].
\end{array}
\]
We thus obtain
\[
\begin{array}{lll}
 \lambda_{1}\left(D_{R}\right) & \leq &
 \fr{\dint_{D} \left|\nabla u\right|^{2}(x,y)\, {\rm d}y{\rm d}x}{\dint_{D} u^{2}(x,y)\, {\rm d}y{\rm d}x }\eqskip
 & = & \fr{\fr {\pi\left (6 + \pi^2 \right)} {\sqrt {R^2 -
 \pi^2}} + 2\left (\pi^2 - 3 \right)\arctan\left (\fr {\pi} {\sqrt {R^2 - \pi^2}} \right)} {6\left[\pi\sqrt {R^2 -
 \pi^2} + R^2\arctan\left (\fr{\pi} {\sqrt {R^2 - \pi^2}} \right) \right]}.
\end{array}
\]
Combining this expression with $4\pi/\left|D_{R}\right|$ we obtain that for values of $R$ larger
than $\pi$ for which the quantity on the right-hand side of the inequality
\[
\begin{array}{lll}
 \fr{4\pi}{\left|D_{R}\right|}-\lambda_{1}\left(D_{R}\right) & \geq &
 \fr{\pi\left (12\sqrt {R^2 - \pi^2} - \pi^2 - 6 \right) - 2\left (\pi^2 - 3 \right)\sqrt {R^2 - \pi^2}
 \arctan\left (\fr {\pi} {\sqrt {R^2 - \pi^2}} \right)} {6\left[ \pi (R^2 -\pi^2)  + \sqrt {R^2 - \pi^2}  R^2
 \arctan\left (\fr {\pi} {\sqrt {R^2 - \pi^2}}  \right) \right]}
\end{array}
\]
is positive, the first eigenvalue of $D_{R}$ does not satisfy P\'{o}lya's conjecture. It is readily seen that
the denominator above is positive for $R$ larger than $\pi$, and so the sign of the whole expression is given
by the sign of the numerator. This may, in turn, be written in the form
\[
 c_{1} w -c_{2} - c_{3}w\arctan\left({\fr{\pi}{w}}\right),
\]
where $w=\sqrt{R^2-\pi^2}$ and $c_{i}$, $i=1,2,3$ are positive constants. It is not difficult to see that (for
these values of the constants $c_{i}$) this is an increasing function of $w$, which is negative when $w$ vanishes
and goes to plus infinity with $w$. The given value for $R_{1}$ may be obtained by solving (numerically) the equation
resulting from equating the numerator above to zero.
\end{proof}

\subsection{Higher dimensions: some estimates for general cylinders\label{higher}} Although similar studies may be
carried out for more general cylinders such as those of the form $\mathbb{S}^{n}\times I_{h}$ and
$\mathbb{S}^{1}\times {\ds\prod_{j=1}^{n}} I_{h_{j}}$, the corresponding analysis becomes much more involved and we
will thus focus on two aspects which we believe already provide a global illustration of what is to be expected. On the
one hand, it is relatively straightforward to consider again the first eigenvalue by itself and see when P\'{o}lya's
conjecture fails in this case. We expect that in the former case, namely $\mathbb{S}^{n}\times I_{h}$, the values of
$h$ where P\'{o}lya's conjecture breaks down for $\lambda_{1}$ will decrease with $n$, while in the latter case
increasing the dimension will allow for higher values of the indices $h_{j}$ for which P\'{o}lya's conjecture holds.

\subsubsection{$\mathbb{S}^{n}\times I_{h}$}

In this instance, Pólya's conjecture reads as
\begin{align*}
    N(\lambda)&\leq \lambda^{\frac{n+1}{2}}\frac{\omega_{n+1}}{(2\pi)^{n+1}}|\mathbb{S}^{n}\times I_{h}|\\
    &= \lambda^{\frac{n+1}{2}}\frac{(n+1)\omega_{n+1}^2}{(2\pi)^{n+1}} h\\
    &= \lambda^{\frac{n+1}{2}}\frac{n+1}{2^{n+1}\Gamma\left(\frac{n+3}{2}\right)^2} h,
\end{align*}
while the first Dirichlet eigenvalue is again given by
$$\lambda_1(\mathbb{S}^{n}\times I_{h})=\frac{\pi^2}{h^2}.$$
The conjecture fails for this eigenvalue when
\begin{align*}
    1&>\left(\frac{\pi^2}{h^2}\right)^{\frac{n+1}{2}}\frac{n+1}{2^{n+1}\Gamma\left(\frac{n+3}{2}\right)^2} h\\
    \Leftrightarrow h &>\pi^{1+\frac{1}{n}}\frac{(n+1)^{\frac{1}{n}}}{2^{1+\frac{1}{n}}\Gamma\left(\frac{n+3}{2}\right)^{\frac{2}{n}}}\\
    &=\frac{\pi}{2}\left(\frac{\pi(n+1)}{2 \Gamma\left(\frac{n+3}{2}\right)^2}\right)^{\frac{1}{n}},
\end{align*}
and we conclude
\begin{proposition}
Pólya's conjecture on $\mathbb{S}^{n}\times I_{h}$ does not hold for $h$ larger than
$$\eta_{1}(n)=\fr{\pi}{2}\left(\frac{\pi(n+1)}{2 \Gamma\left(\frac{n+3}{2}\right)^2}\right)^{\frac{1}{n}}\approx \fr{e \pi}{n}+
\bo\left(\fr{\log(n)}{n^2}\right), \mbox{ as } n\to\infty.$$
\end{proposition}
We thus see that, as $n$ increases, the value of $h$ for which Pólya's conjecture is satisfied will decrease. We
remark that, just like in the case where $n$ is one, there might still exist gaps in the interval $(0,\eta_{1}(n))$
where Pólya's conjecture is not satisfied, similar to those corresponding to $\lambda_8$ and $\lambda_{13}$
for $\SI$.

\subsubsection{$\mathbb{S}^{1}\times {\ds\prod_{j=1}^{n}} I_{h_{j}}$}

P\'{o}lya's conjecture 
now reads as
\begin{align*}
    N(\lambda)&\leq \lambda^{\frac{n+1}{2}}\frac{\omega_{n+1}}{(2\pi)^{n+1}}\left|\mathbb{S}^{1}\times {\ds\prod_{j=1}^{n}} I_{h_{j}}\right|\\
    &= \lambda^{\frac{n+1}{2}}\frac{h_1h_2\cdots h_n}{(2\pi)^n} \omega_{n+1}\\
    &= \lambda^{\frac{n+1}{2}}\frac{h_1h_2\cdots h_n}{2^n\pi^{\frac{n-1}{2}}\Gamma\left(\frac{n+3}{2}\right)}
\end{align*}
We shall again consider the first Dirichlet eigenvalue, now given by
$$\lambda_1\left(\mathbb{S}^{1}\times {\ds\prod_{j=1}^{n}} I_{h_{j}}\right)=\frac{\pi^2}{h_1^2}+\frac{\pi^2}{h_2^2}+\dots+\frac{\pi^2}{h_n^2},$$
and see that the conjecture fails for this eigenvalue when
\begin{align*}
    1&>\left(\frac{\pi^2}{h_1^2}+\frac{\pi^2}{h_2^2}+\dots+\frac{\pi^2}{h_n^2}\right)^{\frac{n+1}{2}}\frac{h_1h_2\cdots h_n}{2^n\pi^{\frac{n-1}{2}}\Gamma\left(\frac{n+3}{2}\right)}\\
    &=\left(h_1^{-2}+h_2^{-2}+\dots+h_n^{-2}\right)^{\frac{n+1}{2}}h_1h_2\cdots h_n\frac{\pi^{\frac{n+3}{2}}}{2^n\Gamma\left(\frac{n+3}{2}\right)}\\
    \Leftrightarrow \frac{\pi^{1/2}}{2}\left(\frac{2\pi}{\Gamma\left(\frac{n+3}{2}\right)}\right)^{\frac{1}{n+1}} &<\left(h_1^{-2}+h_2^{-2}+\dots+h_n^{-2}\right)^{-1/2}(h_1h_2\cdots h_n)^{-\frac{1}{n+1}}.
\end{align*}
We thus have
\begin{proposition}
    Pólya's conjecture on $\mathbb{S}^{1}\times {\ds\prod_{j=1}^{n}} I_{h_{j}}$ may be safisfied only if
    $$\frac{\pi^{1/2}}{2}\left(\frac{2\pi}{\Gamma\left(\frac{n+3}{2}\right)}\right)^{\frac{1}{n+1}} \geq\left(h_1^{-2}+h_2^{-2}+\dots+h_n^{-2}\right)^{-1/2}(h_1h_2\cdots h_n)^{-\frac{1}{n+1}}$$
\end{proposition}

In the particular case of the Cartesian product of $\mathbb{S}^{1}$ with a hypercube of side length $h$, that is
$h_1=h_2=\dots=h_n=h$, this becomes
\begin{align*}
    \frac{\pi^{1/2}}{2}\left(\frac{2\pi}{\Gamma\left(\frac{n+3}{2}\right)}\right)^{\frac{1}{n+1}} &<\left(h_1^{-2}+h_2^{-2}+\dots+h_n^{-2}\right)^{-1/2}(h_1h_2\cdots h_n)^{-\frac{1}{n+1}}\\
    \Leftrightarrow \frac{\pi^{1/2}}{2}\left(\frac{2\pi}{\Gamma\left(\frac{n+3}{2}\right)}\right)^{\frac{1}{n+1}}&<n^{-1/2}h^{\frac{1}{n+1}}\\
    h &> \frac{n^{\frac{n+1}{2}}\pi^{\frac{n+3}{2}}}{2^n\Gamma\left(\frac{n+3}{2}\right)}
\end{align*}
and we get
\begin{corollary}
    Pólya's conjecture on $\mathbb{S}^{1}\times {\ds\prod_{j=1}^{n}} I_{h}$ does not hold if $h$ is larger than
    $$\fr{n^{\frac{n+1}{2}}\pi^{\frac{n+3}{2}}}{2^n\Gamma\left(\frac{n+3}{2}\right)}\approx
    \fr{ \pi}{\sqrt{2n} }\left(\fr{e \pi}{2}\right)^{n/2}+\bo\left(\fr{1}{n^{3/2}}\left(\fr{e \pi}{2}\right)^{n/2}\right), \mbox{ as } n\to\infty.$$
\end{corollary}
\noindent As $n$ increases the values of $h$ for which P\'{o}lya's conjecture may be satisfied increase, but again
there might still exist gaps in this interval where the conjecture fails.

In this instance, and because the Li-Yau inequalities hold for $\mathbb{S}^{1}\times I_{h}$ for $h$
in $(0,\pi^2]$, we may, in fact, conclude that the conjecture will hold provided the smallest of the $h_{i}$
is in this interval.
\begin{proposition} Let $n$ be an integer greater than or equal to three, and
 let $h_{i}$, $i=1,\dots,n$ be positive real numbers, and $h_{0} := \min\left\{h_{1},\dots,h_{n}\right\}$. If
 $h_{0} \in(0,\pi^2]$, then P\'{o}lya's conjecture holds for $\mathbb{S}^{1}\times{\ds\prod_{j=1}^{n}} I_{h}$.
\end{proposition}
\begin{proof}
 We write $\mathbb{S}^{1}\times{\ds\prod_{j=1}^{n}} I_{h}$ as the Cartesian product of
 $\mathbb{S}^{1}\times I_{h_{0}}$
 with the remaining product of intervals. Then, because the former domain satisfies the Li-Yau inequalities while
 the latter satisfies P\'{o}lya's inequality and is in $\R^{d}$ with $d$ greater than or equal to two, we may apply
 the argument of~\cite[Theorem 2.8]{lapt} for Cartesian products, from which the result follows.
\end{proof}
\begin{remark}
If $\mathbb{S}^{1}\times I_{h_{1}}$ satisfies P\'{o}lya's conjecture, we may conclude that
$\mathbb{S}^{1}\times I_{h_{1}}\times I_{h_{2}}$ also satisfies the conjecture.
\end{remark}

\appendix
\section*{Appendices}
The software package \textit{Mathematica} was chosen to run all the computational searches due to its algebraic manipulation capabilities. Below we present the code for each of the three algorithms
used, together with some brief descriptions.
\section{Intersection with $k-$intervals: satisfiability of P\'{o}lya's conjecture}
\label{apx:impl}
 The main idea for an efficient implementation of the search for $k$ intersecting $k-$intervals is algebraically sorting all left and right endpoints of every $k-$interval under consideration (roughly $2\times 10^9$ points) and traversing these from left to right. We then keep count of the number of $k-$intervals containing the current
 point, adding or subtracting one each time we pass a left endpoint or a right endpoint, respectively.
 The concrete implementation used is given below.

\begin{lstlisting}[language=Mathematica]
getKIntervals[k_] := 
 Module[{intstart, intend, ints, cnt, flag, start, kInts},
  intstart = Flatten[Table[
     If[x == 0,
      {y^2*Pi/(4 k), 1},
      If[k^2/Pi^2 >= x^2*y^2,
       {(k/Pi - Sqrt[k^2/Pi^2 - x^2*y^2])/(2*x^2), 2},
       Null
       ]]
     , {x, 0, k }, {y, 1, k}], 1];
  intend = Flatten[Table[
     If[x == 0,
      {Infinity, -1},
      If[k^2/Pi^2 >= x^2*y^2,
       {(k/Pi + Sqrt[k^2/Pi^2 - x^2*y^2])/(2*x^2), -2},
       Null
       ]
      ]
     , {x, 0, k }, {y, 1, k}], 1];
  ints = Join[intstart, intend] /. Null -> Sequence[];
  ints = NumericalSort[ints];
  cnt = 0;
  flag = True;
  kInts = {};
  start = -1;
  Do[
   cnt = cnt + pnt[[2]];
   If[cnt >= k && flag,
    start = pnt[[1]];
    flag = False;
    ,
    If[cnt < k && ! flag,
      AppendTo[kInts, {start, pnt[[1]]}];
      start = -1;
      flag = True;
      ];
    ]
   , {pnt, ints}];
  kInts
  ]

Do[Print[{i, Select[getKIntervals[i], #[[1]] < Pi/4 &]}], {i, 1, 1019}]
\end{lstlisting}

\section{Intersection with $k-$intervals: weaker condition}
Further improvements were needed due to the larger inputs: \verb|Join[]| was used instead of \verb|Flatten[]| and other ways to reduce the number of $k$-intervals were implemented (described in Section \ref{liyau-sec}).

The search space is
\begin{itemize}
    \item $\lambda < 26\,301$
    \item $k \leq 130\, 298$
    \item $|m| \leq 162$
    \item $n \leq 509$
\end{itemize}

\label{apx:impl2}
\begin{lstlisting}[language=Mathematica]
getKIntervalsRelaxed[k_, low_, high_] := 
 Module[{val, int1, int2, int4, ints, cnt, flag, start, kInts},
  cnt = 0;
  int1 = Array[
    Function[y,
     val = y^2*Pi^2/(2 (k-1/2));
     If[val < low,
      cnt = cnt + 1; Null,
      If[val <= high,
       {val, 1},
       Null]
      ]]
    , 509];
  int2 = Join @@ Array[
     Function[{x, y},
      If[(k-1/2)^2 < x^2*y^2*Pi^2 || x^2 + y^2/Pi^2 > 26301,
       Null,
       val = ((k-1/2) - Sqrt[(k-1/2)^2 - x^2*y^2*Pi^2])/(x^2);
       If[val < low,
        cnt = cnt + 2; Null,
        If[val <= high,
         {((k-1/2) - Sqrt[(k-1/2)^2 - x^2*y^2*Pi^2])/(x^2), 2},
         Null
         ]
        ]
       ]
      ]
     , {162 , 509}];
  int4 = Join @@ Array[
     Function[{x, y},
      If[(k-1/2)^2 < x^2*y^2*Pi^2 || x^2 + y^2/Pi^2 > 26301,
       Null,
       val = ((k-1/2) + Sqrt[(k-1/2)^2 - x^2*y^2*Pi^2])/(x^2);
       If[val < low,
        cnt = cnt - 2; Null,
        If[val <= high,
         {((k-1/2) + Sqrt[(k-1/2)^2 - x^2*y^2*Pi^2])/(x^2), -2},
         Null
         ]
        ]
       ]
      ]
     , {162 , 509}];
  ints = (Join @@ {{{low, 0}}, int1, int2, 
       int4, {{high, -Infinity}}}) /. Null -> Sequence[];
  ints = NumericalSort[ints];
  flag = True;
  kInts = {};
  start = -1;
  Do[
   cnt = cnt + pnt[[2]];
   If[cnt >= k && flag,
    start = pnt[[1]];
    flag = False;
    ,
    If[cnt < k && ! flag,
      AppendTo[kInts, {start, pnt[[1]]}];
      start = -1;
      flag = True;
      ];
    ]
   , {pnt, ints}];
  kInts
  ]

Do[Print[{i, getKIntervals2[i, Pi^2/2, Pi^2]}], {i, 1, 130298}]
\end{lstlisting}

\section{Algorithm to check the Li-Yau inequalities}
\label{apx:impl3}
As described in Section~\ref{liyau-sec}, to check the Li-Yau inequalities we first compute all the intersections between eigenvalues (as functions of $h$) and sort them to form a partition of $(\pi^2/2,\pi^2]$. Then for each interval we calculate the exact values for the first $k=86$ eigenvalues and determine if Li-Yau is satisfied for the whole interval.

\begin{lstlisting}[language=Mathematica]
liYauInRange[k_, h0_, h1_] := 
 Module[{poi, ioi, ibeg, iend, hi, eigs, sm, sn, res},
  poi = Flatten[Table[
      If[Or[And[y1 > y0, x1 < x0], And[y1 < y0, x1 > x0]], 
       Sqrt[(y1^2 - y0^2)/(x0^2 - x1^2)]*Pi]
      , {x0, 0, 9 }, {y0, 1, 28}, {x1, 0, 9}, {y1, 0, 28}], 3] /. 
    Null -> Sequence[];
  poi = DeleteDuplicates[poi];
  poi = Table[
     If[And[v > h0, v < h1], v]
     , {v, poi}] /. Null -> Sequence[];
  PrependTo[poi, h0];
  AppendTo[poi, h1];
  poi = NumericalSort[poi];
  ioi = Table[{poi[[i]], poi[[i + 1]]}, {i, 1, Length[poi] - 1}];
  res = {};
  Do[
   ibeg = hh[[1]];
   iend = hh[[2]];
   hi = (ibeg + iend)/2;
   eigs = 
    Join @@ Table[{x^2 + y^2*Pi^2/hi, {x, y}}, {x, -9, 9}, {y, 1, 28}];
   eigs = NumericalSort[eigs];
   sm = 0; sn = 0;
   Do[
    sm = sm + eigs[[ki]][[2]][[1]]^2;
    sn = sn + Pi^2*eigs[[ki]][[2]][[2]]^2;
    If[sm == 0,
     If[iend > sn/(ki^2),
      AppendTo[res, {hh, ki}]
      ]
     ,
     If[ki^4 >= 4*sm*sn,
      If[(ki^2 + Sqrt[ki^4 - 4*sm*sn])/(2*sm) > 
         ibeg && (ki^2 - Sqrt[ki^4 - 4*sm*sn])/(2*sm) < iend,
       AppendTo[res, {hh, ki}]
       ]
      ]
     ];
    , {ki, 1, k}];
   , {hh, ioi}];
  res
  ]

liYauInRange[86, Pi^2/2, Pi^2]
\end{lstlisting}

\section*{Acknowledgements}
The first author was partially supported by the Funda\c c\~{a}o para a Ci\^{e}ncia e a Tecnologia, I.P. (Portugal), through project {\small doi.org/10.54499/UIDB/00208/2020}. The second author was partially supported by the Funda\c c\~{a}o Calouste Gulbenkian (Portugal) under the program {\it
Novos Talentos}.

\end{document}